\documentclass{amsart}

\usepackage{macros}
\standardmargins\standardpackages\standardrenews\standardlabeling\standardmakes
\colorcommentstrue

\newcommand{\ic}{\mathrm i}

\draftfalse

\begin{document}
\title{Some manifolds of Khinchin type for convergence}

\authordavid

%\subjclass[2010]{Primary }
%\keywords{}
%\date{}
%\dedicatory{}

\begin{Abstract}
Recently, Beresnevich, Vaughan, Velani, and Zorin (preprint '15) gave some sufficient conditions for a manifold to be of Khinchin type for convergence. We show that their techniques can be used in a more optimal way to yield stronger results. In the process we also improve a theorem of Dodson, Rynne, and Vickers ('89).
\end{Abstract}
\maketitle

\section{Khinchin-type results}
\label{sectionkhinchin}
Fix $n\in\N$. It is an easy consequence of the Borel--Cantelli lemma that if $\psi:\N\to \Rplus$ is a function such that the series
\begin{equation}
\label{khinchinseries}
\sum_{q = 1}^\infty \psi^n(q)
\end{equation}
converges, then for all $\bftheta\in \R^n$, the set
\[
\SS(\psi,\bftheta) \df \{\xx\in \R^n : \exists^\infty q\in\N \;\; \|q\xx - \bftheta\| < \psi(q)\}
\]
is of Lebesgue measure zero. Here $\|\cdot\|$ denotes distance to the nearest integer vector, measured using the max norm, which we denote by $|\cdot|$. The preceding result is known as \emph{the convergence case of Khinchin's theorem}. A manifold $\MM \subset \R^n$ is said to be of Khinchin type for convergence if its typical points behave like the typical points of Lebesgue measure with respect to this theorem. More precisely, let us say that $\MM$ is \emph{of strong (resp. weak) Khinchin type for convergence} if for every function (resp. monotonic function) $\psi$ satisfying \eqref{khinchinseries} and for all $\bftheta\in \R^n$, the set $\SS(\psi,\bftheta)\cap M$ has measure zero with respect to the Lebesgue measure of $\MM$.\Footnote{The use of the adjectives ``strong'' and ``weak'' in this context is new. In the literature, the phrase ``Khinchin type for convergence'' usually means ``weak Khinchin type for convergence''.}

Recently, Beresnevich, Vaughan, Velani, and Zorin proved the following theorem (which we have taken some liberties in rephrasing):

\begin{theorem}[{\cite[Corollaries 3 and 5]{BVVZ}}]
\label{theoremBVVZ}
Fix $d,m\in\N$, let $\KK\subset\R^d$ be a closed rectangle,\Footnote{Although the authors of \cite{BVVZ} only prove the special case $\KK = [0,1]^d$, their arguments work just as well for the general case.} let $\ff:\KK\to\R^m$ be a function of class $\CC^2$, and let
\begin{equation}
\label{Mfdef}
\MM = \MM_{\KK,\ff} = \{(\bfalpha,\ff(\bfalpha)) : \bfalpha\in \KK\} \subset \R^{d + m}.
\end{equation}
If we are in either of the following scenarios:
\begin{itemize}
\item[1.] $m < d - 1$, and for Lebesgue-a.e. $\bfalpha\in \KK$, we have
\begin{equation}
\label{det1}
\det\big(f_j''(\bfalpha)[\ee_1,\ee_i]\big)_{1 \leq i,j \leq m} \neq 0;
\end{equation}
\item[2.] $m = 1$, $d\geq 2$, and for Lebesgue-a.e. $\bfalpha\in \KK$, we have
\begin{equation}
\label{det2}
\det\big(f_1''(\bfalpha)[\ee_i,\ee_j]\big)_{1 \leq i,j \leq d} \neq 0;
\end{equation}
\end{itemize}
then $\MM$ is of strong Khinchin type for convergence.\Footnote{Although the statements of \cite[Corollaries 3 and 5]{BVVZ} only yield that $\MM$ is of weak Khinchin type for convergence, the proofs actually show that $\MM$ is of strong Khinchin type for convergence, since the assumption that $\psi$ is monotonic is not used anywhere in the proofs.}
\end{theorem}

Here $(\ee_i)_{1 \leq i \leq d}$ denotes the standard basis of $\R^d$, and $f_j$ denotes the $j$th component of $\ff$.

An important fact about this theorem, which indicates that it is ``well phrased'', is that the hypotheses \eqref{det1} and \eqref{det2} are \emph{satisfiable} in the following sense: For any three numbers $d,m,n\in\N$ satisfying $d + m = n$ as well as the appropriate numerical hypothesis/hypotheses (i.e. $m < d - 1$ for Case 1, and $m = 1$, $d\geq 2$ for Case 2), there exists a function (and in fact many functions) $\ff:\KK\to\R^m$ such that the appropriate hypothesis on $\ff''$ (i.e. \eqref{det1} for Case 1, and \eqref{det2} for Case 2) holds. This indicates that the theorem is non-vacuous in a ``uniform'' way. Although this observation is somewhat trivial in the case of Theorem \ref{theoremBVVZ}, it will be less trivial in the case of the next two theorems.

Theorem \ref{theoremBVVZ} bears a strong resemblance to a theorem of Dodson, Rynne, and Vickers, which for convenience we write in a similar format:

\begin{theorem}[{\cite[Theorem 1.3]{DRV}}]
\label{theoremDRV}
Fix $d,m\in\N$, let $\KK \subset \R^d$ be a closed rectangle, let $\ff:\KK\to\R^m$ be a function of class $\CC^2$, and let $\MM \subset \R^{d + m}$ be as in \eqref{Mfdef}. If
\begin{itemize}
\item $m \leq \binom{d}{2}$, and for Lebesgue-a.e. $\bfalpha\in\KK$,% for all $\tt\in\R^m\butnot\{\0\}$, the matrix
%\begin{equation}
%\label{DRVmatrix}
%\big(\tt\cdot\ff''(\bfalpha)[\ee_i,\ee_j]\big)_{1\leq i,j\leq d}
%\end{equation}
%has at least two nonzero eigenvalues that share the same sign,
\begin{equation}
\label{DRVcond}
\begin{split}
&\text{for all $\tt\in\R^m\butnot\{\0\}$, the matrix $\big(\tt\cdot\ff''(\bfalpha)[\ee_i,\ee_j]\big)_{1\leq i,j\leq d}$}\\
&\text{has at least two nonzero eigenvalues that share the same sign.}
\end{split}
\end{equation}
\end{itemize}
then $\MM$ is of strong Khinchin type for convergence.
\end{theorem}

Here the assumption $m \leq \binom{d}{2}$ does not appear in \cite{DRV}, but we have added it because the hypothesis \eqref{DRVcond} is not satisfiable when $m > \binom{d}{2}$ (Theorem \ref{theoremtypical}(v)). It appears to be a difficult problem to determine precisely for which values $m \leq \binom{d}{2}$ the hypothesis is satisfiable; cf. the discussions in \cite[\62]{DRV2} and in Section \ref{sectiontypical}. We have made progress on this problem by showing that the hypothesis is satisfied generically whenever $m \leq \binom{d - 1}{2}$ (Theorem \ref{theoremtypical}(vii)).

The main goal of this paper is to generalize Theorems \ref{theoremBVVZ} and \ref{theoremDRV} simultaneously, yielding a new theorem more powerful than both of them. In the following theorem, Case 1 is a generalization of Case 1 of Theorem \ref{theoremBVVZ}, and Case 2 is a generalization of Case 2 of Theorem \ref{theoremBVVZ} and also of Theorem \ref{theoremDRV}:

\begin{theorem}
\label{theoremkhinchin}
Fix $d,m\in\N$, let $\KK\subset\R^d$ be a closed rectangle, let $\ff:\KK\to\R^m$ be a function of class $\CC^2$, and let the manifold $\MM \subset \R^{d + m}$ be given by \eqref{Mfdef}. If we are in either of the following scenarios:
\begin{itemize}
\item[1.] $m < d - 1$, and for Lebesgue-a.e. $\bfalpha\in \KK$,
\begin{equation}
\label{surjective}
\text{the map $\ff''(\bfalpha):\Sym^2\R^d \to \R^m$ is surjective;}
\end{equation}
\item[2.] $m < \binom{d + 1}{2}$, and for Lebesgue-a.e. $\bfalpha\in \KK$, we have
\begin{equation}
\label{rank2}
\rank\big(\tt\cdot\ff''(\bfalpha)[\ee_i,\ee_j]\big)_{1 \leq i,j \leq d} \geq 2 \all \tt\in\R^m\butnot\{\0\};
\end{equation}
\end{itemize}
then $\MM$ is of strong Khinchin type for convergence.
\end{theorem}

We now show that this theorem is in fact a generalization of Theorems \ref{theoremBVVZ} and \ref{theoremDRV}:
\begin{proof}[Proof that Theorem \ref{theoremkhinchin} implies Theorem \ref{theoremBVVZ}]
The linear transformation $\ff''(\bfalpha):\Sym^2\R^d\to \R^m$ is surjective if and only if some $m\times m$ minor of its corresponding matrix has a nonzero determinant. Since the matrix on the left-hand side of \eqref{det1} is such a minor (since it is the matrix corresponding to the linear transformation $\ff''(\bfalpha)\given \sum_{i = 1}^m \R\ee_1 \ee_m$), \eqref{det1} implies \eqref{surjective}, and thus Case 1 of Theorem \ref{theoremBVVZ} is a special case of Case 1 of Theorem \ref{theoremkhinchin}.

Suppose that $m = 1$ and $d\geq 2$, and that \eqref{det2} holds. Then for all $\tt\in\R^m\butnot\{\0\}$ we have
\begin{align*}
\rank\big(\tt\cdot\ff''(\bfalpha)[\ee_i,\ee_j]\big)_{1\leq i,j\leq d}
&= \rank\big(f_1''(\bfalpha)[\ee_i,\ee_j]\big)_{1\leq i,j\leq d} \since{$m = 1$}\\
&= d \by{\eqref{det2}}\\
&\geq 2, \by{hypothesis}
\end{align*}
i.e. \eqref{rank2} holds. Thus, Case 2 of Theorem \ref{theoremBVVZ} is a special case of Case 2 of Theorem \ref{theoremkhinchin}.
\end{proof}

\begin{proof}[Proof that Theorem \ref{theoremkhinchin} implies Theorem \ref{theoremDRV}]
The rank of a symmetric matrix is equal to the number of nonzero eigenvalues it has. So if a matrix has at least two nonzero eigenvalues, then its rank is at least two, regardless of the sign of the eigenvalues. Thus \eqref{DRVcond} implies \eqref{rank2}, and so Theorem \ref{theoremDRV} is a special case of Case 2 of Theorem \ref{theoremkhinchin}.
\end{proof}

\begin{remark}
As in Theorem \ref{theoremBVVZ}, the hypotheses \eqref{surjective} and \eqref{rank2} are satisfiable whenever their numerical requirements are satisfied. For \eqref{surjective}, this is an immediate consequence of the implication \eqref{det1} \implies \eqref{surjective} (but see the next remark for an example that satisfies \eqref{surjective} but not \eqref{det1}). It is a little harder to see why \eqref{rank2} is satisfiable for all $m < \binom{d + 1}{2}$; we refer to Section \ref{sectiontypical} for details, specifically Theorem \ref{theoremtypical}(ii). In Section \ref{sectiontypical}, we also show that to enforce that ``almost all'' functions $\ff$ satisfy \eqref{rank2}, the stronger inequality $m \leq \binom{d}{2}$ is needed (Theorem \ref{theoremtypical}(iii,iv)).
\end{remark}

\begin{remark}
\label{remarkdet1}
It should be noted that while the condition \eqref{det1} is not invariant under affine changes of coordinates, the conditions \eqref{det2}, \eqref{surjective}, and \eqref{rank2} are. Also, there are functions $\ff$ satisfying \eqref{surjective} that do not satisfy \eqref{det1} with respect to any affine coordinate system; for example, if $m = 3$ and $d = 5$, then the function
\[
\ff(\alpha_1,\ldots,\alpha_5) = (\alpha_1^2,\alpha_1\alpha_2,\alpha_2^2)
\]
has this property.
\end{remark}

\begin{remark}
Theorems \ref{theoremBVVZ}, \ref{theoremDRV}, and \ref{theoremkhinchin} appear to be the only known results regarding manifolds of strong Khinchin type for convergence (for simultaneous approximation). However, there are some results regarding manifolds of \emph{weak} Khinchin type for convergence; for example, it was proven in \cite[Theorem 5]{VaughanVelani} (see also \cite{BVV}) that nondegenerate planar curves are of weak Khinchin type for convergence. The results of this paper do not apply to planar curves, since the parameters $d = m = 1$ do not satisfy the dimension constraints. It remains an open question whether nondegenerate planar curves (e.g. the standard parabola $\{(x,x^2):x\in\R\}$) are of strong Khinchin type for convergence.
%Regarding dual approximation, a result similar to Case 2 of Theorem \ref{theoremkhinchin} was proven by Dodson, Rynne, and Vickers \cite[Theorem 1.1]{DRV}.
\end{remark}

{\bf Outline of the paper.}
In Sections \ref{sectionjarnik} and \ref{sectioncounting} we continue to state our main results, each time reducing the main result of the previous section to the main result of the current section. Then in Section \ref{sectionproof} we prove the main result of Section \ref{sectioncounting}, and thus by implication all of the main results, using technical tools from \cite{BVVZ}. In Section \ref{sectiontypical} we discuss the significance of the hypothesis \eqref{rank2}, answering the question of how commonly it is satisfied.

In what follows, we do not give an exhaustive comparison of our results with the corresponding results in \cite{BVVZ} and \cite{DRV}; comparing Theorems \ref{theoremBVVZ} and \ref{theoremDRV} vs. Theorem \ref{theoremkhinchin} illustrates the main differences. However, we do make the observation that \cite[(5.1)]{DRV} can be interpreted as a counting result similar to our Theorem \ref{theoremcounting} and \cite[Theorems 1 and 3]{BVVZ}, though it is not phrased in the same language. Standard techniques would then yield a Jarn\'ik-type theorem which could then be compared with Theorem \ref{theoremjarnik} and \cite[Corollaries 3 and 5]{BVVZ}. We leave the details to the interested reader.\\

{\bf Acknowledgements.} The author was supported by the EPSRC Programme Grant EP/J018260/1. The author thanks Victor Beresnevich, Sanju Velani, and Evgeniy Zorin for helpful discussions.

\section{A Jarn\'ik-type result}
\label{sectionjarnik}
The \emph{Hausdorff--Cantelli lemma} \cite[Lemma 3.10]{BernikDodson} is a generalization of the Borel--Cantelli lemma and states that if $g$ is a dimension function (i.e. a nondecreasing continuous function such that $\lim_{\rho\to 0} g(\rho) = 0$) and $(B(x_i,\rho_i))_1^\infty$ is a sequence of balls such that the series $\sum_{i = 1}^\infty g(\rho_i)$ converges, then
\[
\HH^g\Big(\limsup_{i\to\infty} B(x_i,\rho_i)\Big) = 0,
\]
where $\HH^g$ denotes Hausdorff measure with respect to the gauge function $g$ (cf. \cite[\64.9]{Mattila}). As a special case, if $\psi:\N\to \Rplus$ is a function such that the series
\begin{equation}
\label{jarnikseries}
\sum_{q = 1}^\infty q^n g\left(\frac{\psi(q)}{q}\right)
\end{equation}
converges, then for all $\bftheta\in\R^n$, we have $\HH^g(\SS(\psi,\bftheta)) = 0$. As in the previous section, we will give a name to those manifolds that ``inherit'' this property from $\R^n$. Precisely, we will say that a manifold $\MM\subset\R^n$ is \emph{of strong (resp. weak) Jarn\'ik type for convergence} with respect to a dimension function $g$ if for every function (resp. monotonic function) $\psi$ such that the series \eqref{jarnikseries} converges, we have
\[
\HH^{\wbar g}(\SS(\psi,\bftheta)\cap \MM) = 0,
\]
where
\[
\wbar g(\rho) \df \frac{g(\rho)}{\rho^m}\cdot
\]
Here $m$ denotes the codimension of $\MM$. Intuitively, a manifold is of Jarn\'ik type for convergence if the ``size of $\SS(\psi,\bftheta)\cap \MM$ relative to $\MM$'' is no bigger than the ``size of $\SS(\psi,\bftheta)$ relative to $\R^n$'', as measured by the dimension function $g$. Note that a manifold is of strong (resp. weak) Khinchin type for convergence if and only if it is of strong (resp. weak) Jarn\'ik type for convergence with respect to the dimension function $g(\rho) = \rho^n$.

\begin{theorem}
\label{theoremjarnik}
Fix $d,m\in\N$, let $n = d + m$, let $\KK\subset\R^d$ be a closed rectangle, let $\ff:\KK\to\R^m$ be a function of class $\CC^2$, and let the manifold $\MM \subset \R^n$ be given by \eqref{Mfdef}. Let $g$ be a dimension function such that $\wbar g$ is increasing, and suppose that for some $k\in\N$, both of the following hold:
\begin{itemize}
\item[(I)] The series
\begin{equation}
\label{gseries}
\sum_{q = 1}^\infty q^n g\left(q^{-1}\Big(q^{-1}\log^2(q)\Big)^{\tfrac{k}{2m + k}}\right)
\end{equation}
converges;
\item[(II)] For $\HH^{\wbar g}$-a.e. $\bfalpha\in\KK$, we have
\begin{equation}
\label{rankk}
\rank\big(\tt\cdot\ff''(\bfalpha)[\ee_i,\ee_j]\big)_{1 \leq i,j \leq d} \geq k \all \tt\in\R^m\butnot\{\0\}.
\end{equation}
\end{itemize}
Then $\MM$ is of strong Jarn\'ik type for convergence with respect to the dimension function $g$.
\end{theorem}
\begin{remark}
If $g(\rho) = \rho^s$ for some $s > 0$, then the series \eqref{gseries} converges if and only if
\begin{equation}
\label{gseriescriterion}
s\left(1 + \frac{k}{2m + k}\right) > n + 1.
\end{equation}
%Some equivalent ways of writing this inequality are
%\[
%s > n - \frac{k(n - 1) - 2m}{2(m + k)}
%\]
%and
%\[
%s > \frac{n + 1}{2} + \frac{m(n + 1)}{2(m + k)}\cdot
%\]
\end{remark}
\begin{proof}[Proof of Theorem \ref{theoremkhinchin} using Theorem \ref{theoremjarnik}]
First note that the case $k = 1$ of \eqref{rankk} is equivalent to \eqref{surjective}; indeed,
\begin{align*}
\text{\eqref{rankk} holds with $k = 1$}
&\;\;\Leftrightarrow\;\; \tt\cdot\ff''(\bfalpha) \neq \0 \all \tt\in\R^m\butnot\{\0\}\\
&\;\;\Leftrightarrow\;\; \text{$(\ff''(\bfalpha))^T$ is injective}\\
&\;\;\Leftrightarrow\;\; \text{$\ff''(\bfalpha)$ is surjective}.\end{align*}
Now let $g(\rho) = \rho^n$, so that $\wbar g(\rho) = \rho^d$. Then in Case 1 (resp. Case 2) of Theorem \ref{theoremkhinchin}, \eqref{rankk} is satisfied with $k = 1$ (resp. $k = 2$). On the other hand,
\begin{align*}
\text{\eqref{gseries} converges}
&\;\;\Leftrightarrow\;\; \text{\eqref{gseriescriterion} holds with $s = n$}
\;\;\Leftrightarrow\;\; n + \frac{nk}{2m + k} > n + 1\\
&\;\;\Leftrightarrow\;\; 2m < k(n - 1)
\;\;\Leftrightarrow\;\; \begin{cases}
m < d - 1 & k = 1\\
d > 1 & k = 2
\end{cases}\cdot
\end{align*}
So the convergence of \eqref{gseries} with $k = 1$ (resp. $k = 2$) is guaranteed by the appropriate numerical hypothesis of Case 1 (resp. Case 2) of Theorem \ref{theoremkhinchin}.
\end{proof}

\section{A counting result}
\label{sectioncounting}
The proof of Theorem \ref{theoremjarnik} is based on a counting result that is interesting in its own right. Throughout this section, we fix $d,m\in\N$, a closed rectangle $\KK \subset \R^d$, and a function $\ff:\KK\to\R^m$ of class $\CC^2$. Now for each $q\in\N$, $\kappa > 0$, and $\bftheta\in\R^n$, we write $\bftheta = (\bflambda,\bfgamma)\in \R^d\times\R^m$, we consider the set
\begin{align*}
&\RR(q,\kappa,\bftheta)
= \RR_{\KK,\ff}(q,\kappa,\bftheta)
\df \left\{(\aa,\bb)\in\Z^d\times\Z^m :
\dfrac{\aa + \bflambda}{q}\in \KK,\;\;
\left|\ff\left(\dfrac{\aa + \bflambda}{q}\right) - \frac{\bb + \bfgamma}{q}\right| < \frac{\kappa}{q}
\right\},
\end{align*}
and we let $A(q,\kappa,\bftheta) = A_{\KK,\ff}(q,\kappa,\bftheta) = \#\RR(q,\kappa,\bftheta)$.

\begin{convention*}
The notation $A\lesssim B$ means that there exists a constant $C \geq 1$ (the \emph{implied constant}), depending only on universal variables such as $d$, $m$, $\KK$, and $\ff$ (but not on $q$, $\kappa$, and $\bftheta$), such that $A\leq C B$. The notation $A \asymp B$ means $A \lesssim B \lesssim A$. The notation $A \asymp_\plus B$ means that there exists an implied constant $C \geq 0$ such that $A - C \leq B \leq A + C$. 
\end{convention*}

\begin{theorem}
\label{theoremcounting}
Fix $k\in\N$, and suppose that \eqref{rankk} holds for all $\bfalpha\in \KK$. Then for all $q\in\N$, $\kappa > 0$, and $\bftheta\in\R^n$, we have
\begin{equation}
\label{counting}
A(q,\kappa,\bftheta) \lesssim q^d \max(\kappa,\phi(q))^m,
\end{equation}
where $\phi(q) = (q^{-1} \Log^2(q))^{\tfrac{k}{2m + k}}$. Here and hereafter we use the notation
\[
\Log(q) \df \max(1,\log(q)).
\]
\end{theorem}
\begin{proof}[Proof of Theorem \ref{theoremjarnik} using Theorem \ref{theoremcounting}]
First consider the case where \eqref{rankk} holds for all $\bfalpha\in \KK$. Then there exist a rectangle $\LL\subset\R^d$ whose interior contains $\KK$ and an extension of $\ff$ to $\LL$ such that \eqref{rankk} holds for all $\bfalpha\in\LL$. Let $C_1 = 1 + \max_{\bfalpha\in \LL} |\ff'(\bfalpha)|$, and fix $\psi$ such that \eqref{jarnikseries} converges. It is not hard to see that
\[
\SS(\psi,\bftheta) \cap \MM \subset \limsup_{q\to\infty} \bigcup_{(\aa,\bb)\in \RR_{\LL,\ff}(q,C_1\psi(q),\bftheta)} B\left(\left(\frac{\aa + \bflambda}{q},\frac{\bb + \bfgamma}{q}\right),\frac{\psi(q)}{q}\right),
\]
where the ball is taken with respect to the max norm. So by the Hausdorff--Cantelli lemma, if the series
\[
\sum_{q = 1}^\infty A_{\LL,\ff}(q,C_1\psi(q),\bftheta) \; \wbar g\left(\frac{\psi(q)}{q}\right)
\]
converges then $\HH^{\wbar g}(\SS(\psi,\bftheta)\cap \MM) = 0$. And indeed, by Theorem \ref{theoremcounting},
\begin{align*}
\sum_{q = 1}^\infty A_{\LL,\ff}(q,C_1\psi(q),\bftheta) \; \wbar g\left(\frac{\psi(q)}{q}\right)
&\lesssim \sum_{q = 1}^\infty q^d \max(C_1 \psi(q),\phi(q))^m \; \wbar g\left(\frac{\psi(q)}{q}\right)\\
&\lesssim\sum_{q = 1}^\infty q^n \max\left(\frac{\psi(q)}{q},\frac{\phi(q)}{q}\right)^m \wbar g\left(\max\left(\frac{\psi(q)}{q},\frac{\phi(q)}{q}\right)\right)\\
&\leq \sum_{q = 1}^\infty q^n \left[g\left(\frac{\psi(q)}{q}\right) + g\left(\frac{\phi(q)}{q}\right)\right]\\
&\asymp_\plus \eqref{jarnikseries} + \eqref{gseries} < \infty,
\end{align*}
which completes the proof in this case.

%The next part is the same as \cite[Step 2 on p.14]{BVVZ}, but we include it for completeness. 
%For the general case, by contradiction suppose that $\HH^{\wbar g}(\SS(\psi,\bftheta)\cap \MM) > 0$, and let $\bfalpha$ be a point satisfying \eqref{rankk} such that $(\bfalpha,\ff(\bfalpha))$ is in the topological support of $\mu \df \HH^{\wbar g}\given \SS(\psi,\bftheta)\cap \MM$. Such a point must exist because the topological support of $\mu$ has positive $\HH^{\wbar g}$ measure, whereas the set of points such that \eqref{rankk} fails has zero $\HH^{\wbar g}$ measure. Now let $\LL\subset\KK$ be a rectangular neighborhood of $\bfalpha$ such that \eqref{rankk} holds on $\LL$. By the previous argument $\HH^{\wbar g}(\SS(\psi,\bftheta)\cap \MM_{\LL,\ff}) = 0$, but this contradicts the assumption that $(\bfalpha,\ff(\bfalpha))$ is in the topological support of $\mu$.

For the general case, we proceed to re-use the argument given in \cite[Step 2 on p.17]{BVVZ}: Let $V$ be the set of points $\bfalpha\in\KK$ such that \eqref{rankk} holds. Since $V$ is open, it can be written as the union of countably many rectangles, say $V = \bigcup_{i = 1}^\infty \LL_i$. For each $i$, the previous argument shows that $\HH^{\wbar g}(\SS(\psi,\bftheta)\cap \MM_{\LL_i,\ff}) = 0$. On the other hand, by assumption (II) we have $\HH^{\wbar g}(\MM_{\KK\butnot V,\ff}) = 0$. Taking the union gives $\HH^{\wbar g}(\SS(\psi,\bftheta)\cap \MM_{\KK,\ff}) = 0$.
\end{proof}

\section{Proof of Theorem \ref{theoremcounting}}
\label{sectionproof}

The following lemma is a reformulation of the main technical result of \cite{BVVZ}. We provide the proof for completeness.

\begin{lemma}[Cf. {\cite[(2.27) and (2.28)]{BVVZ}}]
\label{lemmaBVVZ}
Let the notation be as in Theorem \ref{theoremcounting}. Then for $\delta > 0$ sufficiently small, independent of $q$, $\kappa$, and $\bftheta$, we have
\begin{equation}
\label{BVVZv2}
A(q,\kappa,\bftheta) \lesssim q^d \frac{1}{H^m} \sum_{\substack{\hh\in\Z^m \\ |\hh| \leq H}} \int_{\KK} \prod_{i = 1}^d \min\left(1,\frac{1}{r\|\hh\cdot \ff'(\bfalpha)[\ee_i]\|}\right) \; \dee \bfalpha \text{ if $H,r \geq 1$},
\end{equation}
where
\begin{align*}
H &\df \left\lfloor\frac{1}{4\kappa}\right\rfloor,&
r &\df \lfloor (\delta q \kappa)^{1/2}\rfloor.
\end{align*}
\end{lemma}
\begin{proof}
In what follows we assume that $H,r \geq 1$. Let $e$ denote the 1-periodic exponential function $e(x) = \exp(2\pi\ic x)$. We will need the following estimates, valid for all $x\in\R$ and $H\in\N$:
\begin{align} \label{expsumbounds1}
\sum_{h = -H}^H (H - |h|) e(hx)
= \left(\frac{\sin(H\pi x)}{\sin(\pi x)}\right)^2
&\geq \left(\frac{2H}{\pi}\right)^2\Big[\|x\| \leq (2H)^{-1}\Big]\\ \label{expsumbounds2}
\sum_{h = -H}^H e(hx)
= \frac{\sin\big((2H + 1)\pi x\big)}{\sin(\pi x)}
&\leq \min\left(2H + 1,\frac{1}{2\|x\|}\right).
\end{align}
Here, the right-hand side of \eqref{expsumbounds1} is written using Iverson bracket notation. Now let $A:\R^d\to\R^m$ be a linear transformation and fix $\yy\in\R^m$. We have
\begin{align*}
\sum_{\substack{\vv\in\Z^d \\ |\vv| \leq r}} \Big[\|A[\vv] + \yy\| \leq (2H)^{-1}\Big]
&= \sum_{\substack{\vv\in\Z^d \\ |\vv| \leq r}} \prod_{j = 1}^m \Big[\|\ee_j\cdot (A[\vv] + \yy)\| \leq (2H)^{-1}\Big]\\
&\lesssim \sum_{\substack{\vv\in\Z^d \\ |\vv| \leq r}} \prod_{j = 1}^m \frac{1}{H} \sum_{h = -H}^H \frac{H - |h|}{H} e\big(h\ee_j\cdot (A[\vv] + \yy)\big) \by{\eqref{expsumbounds1}}\\
&= \frac{1}{H^m}\sum_{\substack{\hh\in \Z^m \\ |\hh| \leq H}} \left(\prod_{j = 1}^m \frac{H - |h_j|}{H}\right) \sum_{\substack{\vv\in\Z^d \\ |\vv| \leq r}} e\big(\hh\cdot (A[\vv] + \yy)\big)\\
&\leq \frac{1}{H^m}\sum_{\substack{\hh\in \Z^m \\ |\hh| \leq H}} \left|\sum_{\substack{\vv\in\Z^d \\ |\vv| \leq r}} e\big(\hh\cdot (A[\vv] + \yy)\big)\right|\\
&= \frac{1}{H^m}\sum_{\substack{\hh\in \Z^m \\ |\hh| \leq H}} \prod_{i = 1}^d \sum_{v = -r}^r e\big(\hh\cdot A[v\ee_i]\big)\\
&\leq \frac{1}{H^m}\sum_{\substack{\hh\in \Z^m \\ |\hh| \leq H}} \prod_{i = 1}^d \min\left(2r + 1,\frac{1}{2\|\hh\cdot A[\ee_i]\|}\right). \by{\eqref{expsumbounds2}}\\
\end{align*}
Now consider a point $\bfalpha\in\KK$, and let $\aa\in\Z^d$ be chosen so that
\begin{equation}
\label{openball}
\frac{\aa + \bflambda}{q} \in \KK \cap B^\circ\left(\bfalpha,\frac{1}{q}\right).
\end{equation}
Here $B^\circ(\bfalpha,1/q)$ denotes the open ball around $\bfalpha$ of radius $1/q$. Such an $\aa$ exists as long as the sides of $\KK$ all have length at least $1/q$, which happens for all sufficiently large $q$. (Small values of $q$ can be dealt with by making $\delta$ smaller.)

Let $\yy = q\ff\big(\frac{\aa + \bflambda}{q}\big)$ and $A = \ff'(\bfalpha)$. Using elementary calculus, one can show that for all $\vv\in\Z^d$ with $|\vv| \leq r$, we have
\[
\left|q\ff\left(\frac{\aa + \vv + \bflambda}{q}\right) - (A[\vv] + \yy)\right| \lesssim \frac{r^2}{q} \leq \delta\kappa,
\]
assuming that $\frac{\aa + \vv + \bflambda}{q} \in \KK$. So if $\delta$ is chosen small enough (depending on $\ff$), then
\[
\left|q\ff\left(\frac{\aa + \vv + \bflambda}{q}\right) - (A[\vv] + \yy)\right| \leq \kappa,
\]
and thus since $2\kappa \leq (2H)^{-1}$,
\begin{equation}
\label{whereweleftoff}
\sum_{\substack{\vv\in \Z^d \\ \vv \leq r}} \left[\left\|q\ff\left(\frac{\aa + \vv + \bflambda}{q}\right)\right\| \leq \kappa\right] \lesssim \frac{1}{H^m}\sum_{\substack{\hh\in \Z^m \\ |\hh| \leq H}} \prod_{i = 1}^d \min\left(2r + 1,\frac{1}{2\|\hh\cdot \ff'(\bfalpha)[\ee_i]\|}\right).
\end{equation}
Now let
\[
\SS(q,\kappa,\bftheta) = \left\{\frac{\aa + \bflambda}{q} : (\aa,\bb)\in\RR(q,\kappa,\bftheta)\right\}.
\]
Since by assumption $H\geq 1$, we have $\kappa \leq 1/4$ and thus $A(q,\kappa,\bftheta) = \#\SS(q,\kappa,\theta)$. On the other hand, for all $\vv\in\Z^d$ such that $\frac{\aa + \vv + \bflambda}{q} \in B(\bfalpha,r/q)$, \eqref{openball} implies that $|\vv| < r + 1$ and thus that $|\vv| \leq r$. Thus, \eqref{whereweleftoff} implies that
\[
\#\big(B(\bfalpha,r/q)\cap\SS(q,\kappa,\bftheta)\big) \lesssim r^d \frac{1}{H^m}\sum_{\substack{\hh\in \Z^m \\ |\hh| \leq H}} \prod_{i = 1}^d \min\left(1,\frac{1}{r\|\hh\cdot \ff'(\bfalpha)[\ee_i]\|}\right).
\]
Integrating over all $\bfalpha\in\KK$ gives
\[
\sum_{\bfbeta\in \SS(q,\kappa,\bftheta)} \lambda\big(\KK\cap B(\bfbeta,r/q)\big) \lesssim r^d \frac{1}{H^m}\sum_{\substack{\hh\in \Z^m \\ |\hh| \leq H}} \int_\KK \prod_{i = 1}^d \min\left(1,\frac{1}{r\|\hh\cdot \ff'(\bfalpha)[\ee_i]\|}\right) \;\dee\bfalpha,
\]
where $\lambda$ denotes Lebesgue measure. Since $\lambda\big(\KK\cap B(\bfbeta,r/q)\big) \asymp (r/q)^d$ for all $\bfbeta\in \SS(q,\kappa,\bftheta)$, rearranging completes the proof.
\end{proof}

We are now ready to prove Theorem \ref{theoremcounting}:

\begin{proof}[Proof of Theorem \ref{theoremcounting} using Lemma \ref{lemmaBVVZ}]
Let $\Omega = \KK \times \del [-1,1]^m$. For each $(\bfalpha,\tt) \in \Omega$ and $I,J \subset \{1,\ldots,d\}$ such that $\#(I) = \#(J) = k$, let $M_{I,J}(\bfalpha,\tt)$ denote the $k\times k$ matrix
\[
M_{I,J}(\bfalpha,\tt) \df \big(\tt\cdot\ff''(\bfalpha)[\ee_i,\ee_j]\big)_{i\in I,\,j\in J},
\]
i.e. $M_{I,J}(\bfalpha,\tt)$ is the $k\times k$ minor of the $d\times d$ matrix $\big(\tt\cdot\ff''(\bfalpha)[\ee_i,\ee_j]\big)_{1\leq i,j\leq d}$ for which $I$ is the set of retained rows and $J$ is the set of retained columns.

Now fix $(\bfalpha_0,\tt_0) \in \Omega$. By \eqref{rankk}, there exist $I,J \subset \{1,\ldots,d\}$ with $\#(I) = \#(J) = k$ such that $\det(M_{I,J}(\bfalpha_0,\tt_0)) \neq 0$. Let $C(\bfalpha_0,\tt_0)$ be a convex neighborhood of $M_{I,J}(\bfalpha_0,\tt_0)$ on which the determinant function is bounded away from zero. Since $M_{I,J}(\bfalpha,\tt)$ depends continuously on $(\bfalpha,\tt)$, there exists a neighborhood $U = U(\bfalpha_0,\tt_0) \subset \Omega$ of $(\bfalpha_0,\tt_0)$ such that for all $(\bfalpha,\tt) \in U$, we have
\begin{equation}
\label{MIJbounds}
M_{I,J}(\bfalpha,\tt) \in C(\bfalpha_0,\tt_0).
\end{equation}
Without loss of generality, we may assume that $U$ is of the form $U = V_1\times\cdots\times V_d\times W$, where $V_i = V_i(\bfalpha,\tt) \subset \KK_i$ and $W = W(\bfalpha,\tt) \subset \del [-1,1]^m$. Here $(\KK_i)_1^d$ denote the factors of $\KK$, so that $\KK = \KK_1\times\cdots\times\KK_d$. We can also assume that the sets $V_1,\ldots,V_d$ are intervals. Now since $\Omega = \KK\times \del [-1,1]^m$ is compact, there exists a finite set $F\subset \Omega$ such that the collection $\{U(\bfalpha_0,\tt_0) : (\bfalpha_0,\tt_0)\in F\}$ covers $\Omega$.

Now fix $(\bfalpha_0,\tt_0)\in F$, let the notation be as above, and let
\begin{align*}
\what J &\df \{1,\ldots,d\}\butnot J,&
V &\df \prod_{j\in J} V_j,&
\what V &\df \prod_{j\in \what J} V_j.
\end{align*}
Fix $\tt\in W$ and $\what\bfbeta\in \what V$, and consider the map
\[
\bfPhi_{\what\bfbeta} : V \ni \bfbeta \mapsto (\tt \cdot \ff'(\bfbeta,\what\bfbeta) [\ee_i])_{i\in I} \in \R^I.
\]
%where $(\bfbeta,\what\bfbeta)_j \df \begin{cases} \beta_j & j \in J \\ \what\beta_j & j\in \what J\end{cases}$.
By \eqref{MIJbounds} and the convexity of $C(\bfalpha_0,\tt_0)$, the map $\bfPhi_{\what\bfbeta}$ is invertible and its Jacobian determinant is bounded away from zero. So
\[
(\bfPhi_{\what\bfbeta})_*[\lambda_V] \lesssim \lambda_{[-R,R]^I},
\]
where $\lambda_S$ denotes Lebesgue measure on a set $S$, and $R > 0$ is sufficiently large. By integrating with respect to $\what\bfbeta$, we get
\begin{equation}
\label{glowerstar}
\bfPhi_*[\lambda_{V\times \what V}] \lesssim \lambda_{[-R,R]^I},
\end{equation}
where $\bfPhi(\bfalpha) = (\tt \cdot \ff'(\bfalpha) [\ee_i])_{i\in I}$.

Now fix $q\in\N$, $\kappa > 0$, and $\bftheta\in\R^n$, let $\delta > 0$ and $H,r\in\N$ be as in Lemma \ref{lemmaBVVZ}, and assume that $H,r\geq 1$. Fix $\hh\in \Z^m$ such that $0 < |\hh| \leq H$. Let $\eta = |\hh| \geq 1$ and $\tt = \eta^{-1} \hh$, and fix $(\bfalpha_0,\tt_0)\in F$ such that $\tt\in W(\bfalpha_0,\tt_0)$. Letting the notation be as above, we have
\begin{align*}
\int_{V\times \what V} \prod_{i = 1}^d \min\left(1,\frac{1}{r \|\hh \cdot \ff'(\bfalpha) [\ee_i]\|}\right) \;\dee\bfalpha
&\leq \int_{V\times \what V} \prod_{i \in I} \min\left(1,\frac{1}{r \|\eta\tt \cdot \ff'(\bfalpha) [\ee_i]\|}\right) \;\dee\bfalpha \noreason\\
&= \int \prod_{i \in I} \min\left(1,\frac{1}{r \|\eta z_i\|}\right) \;\dee \bfPhi_*[\lambda_{V\times \what V}](\zz) \noreason\\
&\lesssim \int \prod_{i \in I} \min\left(1,\frac{1}{r \|\eta z_i\|}\right) \;\dee\lambda_{[-R,R]^I}(\zz) \by{\eqref{glowerstar}}\\
&= \left(\int_{-R}^R \min\left(1,\frac{1}{r\|\eta z\|}\right) \;\dee z\right)^k \since{$\#(I) = k$}\\
\int_{-R}^R \min\left(1,\frac{1}{r\|\eta z\|}\right) \;\dee z
&= \frac{1}{\eta} \int_{-\eta R}^{\eta R} \min\left(1,\frac{1}{r\|z\|}\right) \;\dee z\\
&\leq \frac{2}{\eta} \int_0^{\lceil \eta R\rceil} \min\left(1,\frac{1}{r\|z\|}\right) \;\dee z\\
&= \frac{4\lceil \eta R\rceil}{\eta} \int_0^{1/2} \min\left(1,\frac{1}{r z}\right)\;\dee z\\
&\asymp \int_0^{1/2} \min\left(1,\frac{1}{r z}\right)\;\dee z \since{$\eta \geq 1$ and $R\asymp 1$}\\
% \;\;\; \since{$\eta \geq 1$ and $R\asymp 1$}\\
&\asymp \frac{\Log(r)}{r}\cdot \since{$r\geq 1$}
\end{align*}
Taking the sum over all $(\bfalpha_0,\tt_0)\in F$ such that $\tt\in W(\bfalpha_0,\tt_0)$ gives
\[
\int_{\KK} \prod_{i = 1}^d \min\left(1,\frac{1}{r \|\hh \cdot \ff'(\bfalpha) [\ee_i]\|}\right) \;\dee\bfalpha \lesssim \left(\frac{\Log(r)}{r}\right)^k.
\]
Summing over all $\hh\in\Z^m$ such that $0 < |\hh| \leq H$ and adding 1 to both sides gives
\[
\sum_{\substack{\hh\in \Z^m \\ |\hh| \leq H}} \int_{\KK} \prod_{i = 1}^d \min\left(1,\frac{1}{r \|\hh \cdot \ff'(\bfalpha) [\ee_i]\|}\right) \;\dee\bfalpha \lesssim 1 + H^m \left(\frac{\Log(r)}{r}\right)^k,
\]
and combining with \eqref{BVVZv2} gives
\[
A(q,\kappa,\bftheta) \lesssim q^d \left(\frac{1}{H^m} + \left(\frac{\Log(r)}{r}\right)^k\right) \text{ if $H,r \geq 1$}.
\]
Now suppose that $\phi(q) \leq \kappa \leq 1/4$. Then, assuming that $q$ is sufficiently large, we have $\delta q \kappa \geq 1$. So $H,r \geq 1$, $H \asymp 1/\kappa$, $r\asymp (q\kappa)^{1/2}$, and $\Log(r) \asymp \Log(q)$, and thus
\begin{equation}
\label{Aqkappabound2}
A(q,\kappa,\bftheta) \lesssim q^d \left(\kappa^m + \left(\frac{\Log(q)}{(q\kappa)^{1/2}}\right)^k\right).
\end{equation}
The inequality $\kappa \geq \phi(q)$ allows us to compare the two terms on the right-hand side of \eqref{Aqkappabound2}:
\begin{align*}
\frac{1}{\kappa^m}\left(\frac{\Log(q)}{(q\kappa)^{1/2}}\right)^k
\leq \frac{1}{\phi^m(q)} \left(\frac{\Log(q)}{(q\phi(q))^{1/2}}\right)^k
= \frac{(q^{-1}\Log^2(q))^{k/2}}{\phi^{m + k/2}(q)} = 1,
\end{align*}
which shows that the right-hand term of \eqref{Aqkappabound2} is smaller than the left-hand term. Thus $A(q,\kappa,\bftheta) \lesssim q^d \kappa^m$, and we have completed the proof in the case $\phi(q) \leq \kappa \leq 1/4$, $q$ sufficiently large.

If $\kappa \geq 1/4$, then trivially $A(q,\kappa,\bftheta) \leq (q + 1)^d \lesssim q^d \kappa^m$. On the other hand, if $\kappa \leq \phi(q)$, then 
 $A(q,\kappa,\bftheta) \leq A(q,\phi(q),\bftheta) \lesssim q^d \phi^m(q)$, assuming $q$ is large enough so that $\phi(q) \leq 1/4$. Thus, \eqref{counting} holds in these cases as well. Finally, if $q$ is bounded, then the right hand side of \eqref{counting} is bounded from below while the right hand side is bounded from above, so \eqref{counting} holds in this case as well.
\end{proof}

\section{Typicality of the condition \eqref{rank2}}
\label{sectiontypical}
The reader may notice that we did not use the hypothesis $m < \binom{d + 1}{2}$ in the proof of Theorem \ref{theoremkhinchin}, Case 2 (except for the trivial application to deduce that $d > 1$), but we have still written it into the theorem. Why? Because, as we show below, if $m \geq \binom{d + 1}{2}$, then it is impossible for the hypothesis \eqref{rank2} to be satisfied, so adding the hypothesis $m < \binom{d + 1}{2}$ does not restrict the generality of our theorem. Conversely, if $m < \binom{d + 1}{2}$, then the set of linear operators $A\in \LL(\Sym^2\R^d,\R^m)$ that satisfy
\begin{equation}
\label{rank2withA}
\rank\big(\tt\cdot A[\ee_i,\ee_j]\big)_{1\leq i,j\leq d} \geq 2 \all \tt\in\R^m\butnot\{\0\}
\end{equation}
(i.e. the analogue of \eqref{rank2} with $\ff''(\bfalpha)$ replaced by $A$) is a nonempty open subset of $\LL(\Sym^2\R^d,\R^m)$, meaning that Theorem \ref{theoremkhinchin} is non-vacuous in this case. Here $\LL(\Sym^2\R^d,\R^m)$ denotes the space of linear transformations from $\Sym^2\R^d$ to $\R^m$.% In the language of Section \ref{sectionkhinchin}, the hypotheses of the theorem are satisfiable if we add $m < \binom{d + 1}{2}$ as a numerical hypothesis, but not if we omit it.

Similar logic applies to the hypothesis $m \leq \binom d2$ of Theorem \ref{theoremDRV}. If it is not satisfied, then it is impossible for the main hypothesis of Theorem \ref{theoremDRV} to be satisfied; precisely, there are no linear operators $A\in\LL(\Sym^2\R^d,\R^m)$ such that
\begin{equation}
\label{DRVcondwithA}
\begin{split}
&\text{for all $\tt\in \R^m\butnot\{\0\}$, the matrix $\big(\tt\cdot A[\ee_i,\ee_j]\big)_{1\leq i,j\leq d}$ has}\\
&\text{at least two nonzero eigenvalues that share the same sign}
\end{split}
\end{equation}
(i.e. the analogue of \eqref{DRVcond} with $\ff'(\bfalpha)$ replaced by $A$). However, in this case the converse is not quite true; cf. Remark \ref{remarkdifficultproblem}. A partial converse that is true is that if $m \leq \binom{d - 1}{2}$, then there exist operators $A$ satisfying \eqref{DRVcondwithA}.

Beyond merely verifying that Theorem \ref{theoremkhinchin} is non-vacuous, we may also ask whether its hypotheses are satisfied for ``typical'' manifolds. If $d > 1$ and $m = \binom{d + 1}{2} - 1$, then we will show that the set of linear operators $A$ that do not satisfy \eqref{rank2withA} contains a nonempty open set, meaning that \eqref{rank2withA} both holds and fails on sets of positive measure. This is not a desirable property for a ``nondegeneracy'' condition, which should hold almost everywhere. It turns out that for \eqref{rank2withA} to hold almost everywhere, the stronger inequality $m \leq \binom{d}{2}$ is required. For \eqref{DRVcondwithA}, the appropriate inequality is $m \leq \binom{d - 1}{2}$.

We summarize the above remarks in the following theorem:

\begin{theorem}
\label{theoremtypical}
Fix $d,m\in\N$, and let $U$ (resp. $\w U$) be the set of all linear transformations $A \in \LL \df \LL(\Sym^2\R^d,\R^m)$ satisfying \eqref{rank2withA} (resp. \eqref{DRVcondwithA}). Then $U$ and $\w U$ are open subsets of $\LL$, and:
\begin{itemize}
\item[(i)] If $m \geq \binom{d + 1}{2}$, then $U$ is empty.
\item[(ii)] If $m < \binom{d + 1}{2}$, then $U$ is nonempty.
\item[(iii)] If $m > \binom{d}{2}$, then $U$ is not dense in $\LL$.
\item[(iv)] If $m \leq \binom{d}{2}$, then $U$ is dense in $\LL$; furthermore, its complement is contained in a proper algebraic subset of $\LL$.
\item[(v)] If $m > \binom{d}{2}$, then $\w U$ is empty.
\item[(vi)] If $m > \binom{d - 1}{2}$, then $\w U$ is not dense in $\LL$.
\item[(vii)] If $m \leq \binom{d - 1}{2}$, then $\w U$ is dense in $\LL$; furthermore, its complement is contained in a proper algebraic subset of $\LL$.
\end{itemize}
\end{theorem}
\begin{remark}
The proof below depends crucially on the right-hand side of \eqref{rank2withA} being 2; it would be interesting to ask what happens if if 2 is replaced by a larger integer.
\end{remark}
\begin{remark}
One might wonder whether knowing that \eqref{rank2withA} or \eqref{DRVcondwithA} holds on a full measure set justifies one in thinking that ``most'' $\CC^2$ functions $\ff:\R^d\to\R^m$ satisfy \eqref{rank2} or \eqref{DRVcond}, respectively. If we required the hypothesis to hold for all $\bfalpha\in\KK$, then we could run into a problem: perhaps the set of counterexamples to \eqref{rank2withA} or \eqref{DRVcondwithA} has positive codimension, but is intersected transversally by some set of the form $\{\ff''(\bfalpha) : \bfalpha\in\KK\}$. Then perturbations of this $\ff$ would fail to satisfy \eqref{rank2} or \eqref{DRVcond} on a nonempty (but positive codimension) set of $\bfalpha\in\KK$. But since the conditions are only required to hold on a set of full Lebesgue measure, this does not cause any problem.
\end{remark}
\begin{proof}
Since in \eqref{rank2withA} and \eqref{DRVcondwithA}, the quantifier ``$\forall \tt\in \R^m\butnot\{\0\}$'' can be replaced by ``$\forall \tt\in S^{m - 1}$'' without affecting the truth values, a standard compactness argument shows that $U$ and $\w U$ are open. We proceed to reduce (i)-(vii) to a series of statements about quadratic forms. For each $A\in\LL$, let
\[
V_A = \{\big(\tt\cdot A[\ee_i,\ee_j]\big)_{1 \leq i,j \leq d} : \tt\in\R^m\} \subset \Sym^2\R^d.
\]
Then $A\in U$ if and only if
\begin{itemize}
\item[(I)] The map $\R^m\ni \tt\mapsto \big(\tt\cdot A[\ee_i,\ee_j]\big)_{1 \leq i,j \leq d} \in V_A$ is injective, and 
\item[(II)] For all $B\in V_A\butnot\{\0\}$, $\rank(B) \geq 2$.
\end{itemize}
Similarly, $A\in \w U$ if and only if (I) holds as well as
\begin{itemize}
\item[(III)] For all $B\in V_A\butnot\{\0\}$, $B$ has at least two nonzero eigenvalues that share the same sign.
\end{itemize}
Now, a nonzero element of $\Sym^2\R^d$ is of rank one if and only if it can be written in the form $\pm \vv^2$, where $\vv\in\R^d\butnot\{\0\}$. Similarly, a nonzero element of $\Sym^2\R^d$ fails to have two nonzero eigenvalues sharing the same sign if and only if it can be written in the form $\vv^2 - \ww^2$, where $\vv,\ww\in\R^d$ and $\vv \neq \pm \ww$. Thus, (II) and (III) are respectively equivalent to:
\begin{itemize}
\item[(II$'$)] For all $\vv\in\R^d\butnot\{\0\}$, $\vv^2 \notin V_A$.
\item[(III$'$)] For all $\vv,\ww\in\R^d$ such that $\vv\neq \pm\ww$, $\vv^2 - \ww^2 \notin V_A$.
\end{itemize}
Now if $m > \binom{d + 1}{2}$, then (I) is not satisfied for any $A\in\LL$, so $U = \w U = \emptyset$ and we are done. Otherwise, let $G_\LL$ be the set of all $A\in\LL$ such that (I) is satisfied, i.e. the set of all surjective transformations from $\Sym^2\R^d$ to $\R^m$. Note that the complement of $G_\LL$ is a proper algebraic subset of $\LL$.

Now let $\ell = \binom{d + 1}{2} - m \geq 0$, and consider the spaces $\Omega = \LL(\Sym^2\R^d,\R^\ell)$ and $G_\Omega = \{\text{surjective elements of $\Omega$}\}$. For each $\omega\in \Omega$, let $W_\omega = \{B\in \Sym^2\R^d : \omega[B] = \0\}$. Then the maps $G_\LL\ni A\mapsto V_A$ and $G_\Omega\ni \omega\mapsto W_\omega$ are both algebraic surjections onto the Grassmanian space $\GG_m(\Sym^2\R^d) = \{\text{$m$-dimensional subspaces of $\Sym^2\R^d$}\}$. Letting
\begin{align*}
U_2 &= \{V\in \GG_m(\Sym^2\R^d) : \all \vv\in \R^d\butnot\{\0\} \;\; \vv^2\notin V\}\\
\w U_2 &= \{V\in \GG_m(\Sym^2\R^d) : \all \vv,\ww \in \R^d \;\; \text{if $\vv\neq\pm\ww$ then $\vv^2 - \ww^2\notin V$}\}\\
U_3 &= \{\omega\in\Omega : \omega[\vv^2]\neq \0 \all \vv\in\R^d\butnot\{\0\}\}\\
&= \{(Q_1,\ldots,Q_\ell) \text{ quadratic forms on $\R^d$}: \forall \vv\in\R^d\butnot\{\0\} \;\exists i = 1,\ldots,\ell \;\; Q_i(\vv) \neq 0\}\\
\w U_3 &= \{\omega\in\Omega : \omega[\vv^2]\neq \omega[\ww^2] \all \vv,\ww\in\R^d \text{ such that $\vv\neq\pm\ww$}\}\\
&= \{(Q_1,\ldots,Q_\ell) \text{ quadratic forms on $\R^d$}: \forall \vv,\ww\in\R^d \text{ if $\vv\neq\pm\ww$ then} \;\exists i = 1,\ldots,\ell \;\; Q_i(\vv) \neq Q_i(\ww)\},
\end{align*}
we have
\begin{align*}
U &= \{A\in G_\LL : V_A \in U_2\},&
U_3 \cap G_\Omega &= \{\omega\in G_\Omega : W_\omega \in U_2\},\\
\w U &= \{A\in G_\LL : V_A \in \w U_2\},&
\w U_3 \cap G_\Omega &= \{\omega\in G_\Omega : W_\omega \in \w U_2\}.
\end{align*}
So to complete the proof, we need to show:
\begin{itemize}
\item[(i$'$)] If $\ell = 0$, then $U_3$ is empty.
\item[(ii$'$)] If $\ell > 0$, then $U_3$ is nonempty.
\item[(iii$'$)] If $\ell < d$, then $U_3$ is not dense in $\Omega$.
\item[(iv$'$)] If $\ell \geq d$, then $U_3$ is dense in $\Omega$; furthermore, its complement is contained in a proper algebraic subset of $\Omega$.
\item[(v$'$)] If $\ell < d$, then $\w U_3$ is empty.
\item[(vi$'$)] If $\ell < 2d - 1$, then $\w U_2$ is not dense in $\GG_m(\Sym^2\R^d)$.
\item[(vii$'$)] If $\ell \geq 2d - 1$, then $\w U_2$ is dense in $\GG_m(\Sym^2\R^d)$; furthermore, its complement is contained in a proper algebraic subset of $\GG_m(\Sym^2\R^d)$.
\end{itemize}
Now (i$'$) is obvious, and (ii$'$) follows from the observation that if $Q_1$ is positive-definite, then $(Q_1,0,\ldots,0) \in U_3$. Intuitively, (iii$'$) and (iv$'$) are true because of ``number of variables'' considerations; the intersection of the zero sets of $\ell$ quadratic forms on $\R^d$ should have dimension $d - \ell$, and so generically, the intersection should be zero-dimensional (i.e. equal to $\{\0\}$) if and only if $\ell \geq d$. We proceed to verify this intuitive idea.

When $\ell = 1$, (iii$'$) can be verified by considering any quadratic form which is neither positive semidefinite nor negative semidefinite, but for the general case a different argument is needed. Suppose that $\ell < d$, and for each $i = 1,\ldots,\ell$ let $Q_i(\xx) = x_i x_d$. For each $(\w Q_1,\ldots,\w Q_\ell)\in \Omega$, consider the map
\[
\bfPhi_{\w Q_1,\ldots,\w Q_\ell} : \R^\ell \ni \xx \mapsto (\w Q_1,\ldots,\w Q_\ell)(x_1,\ldots,x_\ell,0,\ldots,0,1),
\]
and observe that $\bfPhi_{Q_1,\ldots,Q_\ell}$ is the identity map. It follows that small perturbations of this map will contain $\0$ in their range. Thus, if $(\w Q_1,\ldots,\w Q_\ell)$ is sufficiently close to $(Q_1,\ldots,Q_\ell)$, then $\0$ is in the range of $\bfPhi_{\w Q_1,\ldots,\w Q_\ell}$, which implies that $(\w Q_1,\ldots,\w Q_\ell) \notin U_3$. So $(Q_1,\ldots,Q_\ell)$ is in the interior of the complement of $U_3$. This completes the proof of (iii$'$).

Next, let $U_3^\C \subset U_3$ be the set of all $\omega\in\Omega$ such that $\omega[\vv^2] \neq \0$ for all $\vv\in\C^d\butnot\{\0\}$. It follows from standard considerations in algebraic geometry that the set $F_3^\C \df \Omega\butnot U_3^\C$ is an algebraic set.\Footnote{For example, if $R(Q_1,\ldots,Q_\ell)$ denotes the multipolynomial resultant of the homogeneous polynomials $Q_1,\ldots,Q_\ell$, then $F_3^\C = \{\omega\in\Omega : R(\omega) = 0\}$ \cite[Theorem 3.(2.3)]{CLoS2}, and in particular $F_3^\C$ is algebraic.} % A more abstract proof of the algebraicity of $F_3^\C$ can be given using the quantifier elimination theorem for the theory of polynomial equations over an algebraically closed field \cite[??\internal]{FriedJarden}: by this theorem, $F_3^\C$ is Zariski constructible, and any closed Zariski constructible set is Zariski closed, so $F_3^\C$ is Zariski closed i.e. algebraic.
Now suppose that $\ell \geq d$, and for each $i = 1,\ldots,d$ let $Q_i(\xx) = x_i^2$. Then for all $\vv\in\C^d\butnot\{0\}$, we have $v_i \neq 0$ for some $i = 1,\ldots,d$ and thus $Q_i(\vv) = v_i^2 \neq 0$. So $(Q_1,\ldots,Q_d,0,\ldots,0)\in U_3^\C$ and in particular $U_3^\C \neq \emptyset$. Thus $F_3^\C$ is a proper algebraic subset of $\Omega$. Since such a set has dimension strictly less than that of the ambient space, it is nowhere dense and thus $U_3^\C$ (and similarly $U$) is dense. This completes the proof of (iv$'$).

Let $W\subset \R^d$ be a nonempty open set such that $W\cap -W = \emptyset$. Suppose that $\w U_3 \neq \emptyset$, and fix $(Q_1,\ldots,Q_\ell) \in \w U_3$. Then $(Q_1,\ldots,Q_\ell):W\to \R^\ell$ is an injective continuous map. Since such a map cannot be dimension-decreasing, we have $\ell\geq d$. This completes the proof of (v$'$).

Let $S = \{\vv^2 - \ww^2 : \vv,\ww\in\R^d\}$, so that $\w U_2 = \{V\in \GG_m(\Sym^2\R^d) : V\cap S = \{\0\}\}$. Note that $S$ is an irreducible closed semi-algebraic set. To compute the dimension of $S$, we note that for all $\vv,\ww\in\R^d$ and $t\in\R$, we have
\[
(\cosh(t)\vv + \sinh(t)\ww)^2 - (\sinh(t)\vv + \cosh(t)\ww)^2 = \vv^2 - \ww^2,
\]
so the map $\gg:\R^{2d} \ni (\vv,\ww)\mapsto \vv^2 - \ww^2 \in S$ has level sets of dimension at least 1 and thus $\dim(S) \leq 2d - 1$. Conversely, direct computation shows that the kernel of $\gg'(\ee_1,\ee_2)$ is $\R(\ee_2,\ee_1)$, a subspace of dimension 1. So the image of $\gg'(\ee_1,\ee_2)$ has dimension $2d - 1$, and thus $\dim(S) \geq 2d - 1$. So $\dim(S) = 2d - 1$.

Let $B$ be a smooth point of $S$, and let $T_B S$ denote the tangent space of $S$ at $B$. Suppose that $\ell < 2d - 1$. Then there exists a subspace $V_0\in \GG_m(\Sym^2\R^d)$ intersecting $S$ transversely at $B$. Here by ``transversely'' we mean that $V_0 + T_B S = \Sym^2\R^d$; we allow $V_0\cap T_B S$ to be nontrivial, and in fact necessarily $\dim(V_0\cap T_B S) \geq 1$ since $B\in V_0\cap T_B S$. If $V\in\GG_m(\Sym^2\R^d)$ is sufficiently close to $V_0$, then $V\cap S \neq \{\0\}$, so $V\notin \w U_2$. So there is a neighborhood of $V_0$ disjoint from $\w U_2$, proving (vi$'$).

Finally, suppose that $\ell \geq 2d - 1$. For each $B\in S\butnot\{\0\}$, the set $I_B \df \{V\in \GG_m(\Sym^2\R^d) : B\in V\}$ has codimension $\ell$ in $\GG_m(\Sym^2\R^d)$. On the other hand, if $\sim$ denotes the projective equivalence relation (i.e. $B\sim tB$ for all $B\in \Sym^2\R^d\butnot\{\0\}$ and $t\in\R\butnot\{0\}$), then $I_{B_1} = I_{B_2}$ whenever $B_1\sim B_2$. So if
\[
\w F_2 = \GG_m(\Sym^2\R^d)\butnot \w U_2 = \bigcup_{B\in S\butnot\{\0\}} I_B,
\]
then
\[
\codim(\w F_2) \geq \ell - \dim(S/\sim) = \ell - (2d - 2) > 0.
\]
Since $\w F_2$ is semi-algebraic, it has the same dimension as its Zariski closure. Thus the Zariski closure of $\w F_2$ is a proper algebraic subset of $\GG_m(\Sym^2\R^d)$ containing the complement of $\w U_2$, completing the proof of (vii$'$).
\end{proof}

\begin{remark}
Following the logic of the proof of (ii) using the identity matrix as an example of a positive-definite matrix shows that if $m = \binom{d + 1}{2} - 1$, then the function
\[
\ff(\alpha_1,\ldots,\alpha_d) = \left(\alpha_1^2 - \alpha_2^2,\alpha_2^2 - \alpha_3^2,\ldots,\alpha_{d - 1}^2 - \alpha_d^2,\alpha_1\alpha_2,\alpha_1\alpha_3,\ldots,\alpha_{d - 1}\alpha_d\right)
\]
(or more generally any function such that the matrices $\big(f_k''[\ee_i,\ee_j]\big)_{1\leq i,j\leq d}$ ($k = 1,\ldots,m$) are a basis for the space of trace-free symmetric matrices) satisfies \eqref{rank2}.
\end{remark}

\begin{remark}
\label{remarkdifficultproblem}
Let $m = 2$ and $d = 3$, and let $\gamma:\Sym^2\R^3\to\R$ be the map that sends a matrix to its middle eigenvalue (i.e. the eigenvalue which is both second-highest and second-lowest). Then $\gamma$ is continuous, $\gamma(-A) = -\gamma(A)$, and $\gamma(A) = 0$ if and only if $A\in S$. So $\Sym^2\R^3\butnot S$ is split into two disjoint connected components $\{\gamma > 0\}$ and $\{\gamma < 0\}$, symmetric to each other via reflection through the origin. A set split in this way cannot contain any subset of the form $V\butnot\{\0\}$, $V\in\GG_2(\Sym^2\R^3)$. So $\w U_2$ is empty in this case, and thus there are no linear operators $A$ satisfying \eqref{DRVcondwithA}. Since $2 \leq \binom{3}{2}$, this shows that the inequality $m \leq \binom{d}{2}$ is not a sufficient condition for the existence of $A$ satisfying \eqref{DRVcondwithA}. (It is not hard to check that this counterexample has the smallest possible dimensions.) Thus it appears to be a difficult problem to determine necessary and sufficient conditions for the existence of such an $A$.
\end{remark}

\bibliographystyle{amsplain}

\bibliography{bibliography}

\end{document}